\documentclass[12pt,reqno]{amsart}
\usepackage{amsmath,amsfonts,amssymb,amsthm,mathrsfs}
\usepackage[T1]{fontenc}
\usepackage{calligra}
\usepackage{verbatim}
\usepackage{setspace}

\def\R{{\mathbb R}}
\def\C{{\mathbb C}}

\def\N{\mathbb{N}}
\def\Z {\mathbb Z}

\def\F{{\mathbb F}}

\newcommand\Lie[1]{{\mathfrak{#1}}}

\def\ad{\mathop{\rm ad}}

\def\Span#1{\mathop{\rm span}\left\{#1\right\}}

\newlength{\pictht}

\newtheorem{theorem}{\sc Theorem}

\newtheorem{lemma}[theorem]{\sc Lemma}

\theoremstyle{definition}
\newtheorem{ex}[theorem]{ \sc Example}

\begin{document}
\title[Contact and $1$-quasiconformal maps on Carnot groups  ]{Contact and $1$-quasiconformal maps on Carnot groups}
\author{}
\author{Alessandro Ottazzi and Ben Warhurst}

\address{Alessandro Ottazzi, Universit\`a di Milano Bicocca,\\ Milano 20126, Italy}

\email{alessandro.ottazzi@unimib.it}

\address{Ben Warhurst, Universit\`a di Milano Bicocca,\\ Milano 20126, Italy}
\email{benwarhurst68@gmail.com}

\date{}

\thanks{Ben Warhurst was supported by the Institute of Mathematics of the
Polish Academy of Sciences whilst carrying out this
research in the period Jan-Sept of 2009}



\begin{abstract}
We characterise the rigidity of Carnot groups in the class of $C^2$ contact maps in terms of complex characteristics. Furthermore, we obtain a Liouville type theorem for all Carnot groups other than  $\R$ or $\R^2$ which states that $1$-quasiconformal maps form finite dimensional
Lie groups.

\end{abstract}
\maketitle
\tableofcontents
\section{Introduction}
A nilpotent Lie algebra $\mathfrak{n}$ is said to admit an $s$-step stratification if it decomposes as the direct sum of subspaces $ \mathfrak{n}=\mathfrak{g}_{-s} \oplus \cdots \oplus \mathfrak{g}_{-1} $ which satisfy the bracket generating property $\mathfrak{g}_{j-1}=[\mathfrak{g}_{-1},\mathfrak{g}_j]$, where $j=-1,\dots,-s+1$, and $\mathfrak{g}_{-s}$ is contained in the centre $ \Lie{z}( \mathfrak{n})$.

Let $N$ denote the connected, simply connected nilpotent Lie group with stratified Lie algebra $\Lie{n}$. The  sub-bundle $\mathcal{H} \subseteq TN$ obtained by left translating $\Lie{g}_{-1}$  is called the {\it horizontal bundle}.
 Furthermore, if $\Lie{n}$ is equipped with an inner product $\langle\,  ,\,  \rangle$ such that $\Lie{g}_{i} \perp \Lie{g}_{j}$ for every $i \ne j$, then $N$ is called Carnot group.

The inner product on $\Lie{n}$ defines an inner product on the tangent space $T_pN$ at every point $p\in N$ by left translation. In particular, for $X,Y \in T_pN$ we have  $$\langle X  , Y  \rangle_p= \langle (\tau_{p^{-1}})_*(X)  , (\tau_{p^{-1}})_*(Y)  \rangle,$$ where $ \tau_{p^{-1}}$ denotes left translation by $p^{-1} \in N$. A smooth curve $\gamma$ is horizontal if its tangent vectors lie in the horizontal bundle, and the length of a horizontal curve is the integral of the lengths of its tangent vectors relative to the inner product described above.  The  { Carnot-Carath\'eodory distance} or subriemannian distance $d(p,q)$ between points $p$ and $q$, is defined as the infimum of the lengths of all horizontal curves joining  $p$ and $q$. By a theorem of Chow~\cite{chow}, the bracket generating property implies that Carnot groups are horizontally path connected and so $d$ is indeed a metric.

Local diffeomorphisms $f$ of $N$ whose differential $f_*$  preserves the horizontal bundle are called {\it contact maps}.
In our framework we shall always  assume, unless otherwise stated, that contact diffeomorphisms are $C^2$.
The group $N$, or its Lie algebra $\Lie{n}$, is said to be {\it rigid} if the space of contact maps between open domains of $N$  form a finite dimensional space.
We shall call $N$ {\it nonrigid} otherwise.
 The simplest examples of contact maps are left translations and dilations. A dilation by $t$ of an element $X \in \Lie{n}$ is defined by $\delta_t(X)=\sum_{k=1}^s t^k X_{-k}$, and the corresponding map of $N$ induced through the exponential defines dilation on $N$ which we also denote by  $\delta_t$. We also note that left translations are isometries of the subriemannian metric and the dilations are homogeneous in the sense that  $d(\delta_t(p),\delta_t(q))=t d(p,q)$.

The notion of rigidity makes sense with low regularity assumptions. For example in the theory of quasiconformal mappings on Carnot groups, the notion of contact map arises in the class of homeomorphisms with coordinate functions in the horizontal Sobolev space $HW_{loc}^{1,1}(N,\R)$.  However, the authors know no example in which the group is rigid in the class of $C^{\infty}$ maps but nonrigid in $HW_{loc}^{1,1}(N,\R)$. This situation provides significant motivation for studying rigidity and has been remarked upon by Kor{\'a}nyi in his MathSciNet review of \cite{MargMost}, by Heinonen in \cite{Heinonen} and by Tyson in \cite{RNC}.

For  $C^{2}$  maps, it is sufficient to consider the rigidity problem for
contact diffeomorphisms which are induced by $C^\infty$ vector fields. A vector field $U$ is said to be a {\it contact vector field} if its local flow is a $1$-parameter group of contact maps. A contact vector field is characterised by the integrability condition $[U,\mathcal{H}]\subseteq \mathcal{H}$. Every contact map fixing $p$ induces an automorphism on the set $\mathcal{CA}_p$ consisting of germs of contact vector fields at a point $p$. If  $\mathcal{CA}_p$ forms a finite dimensional vector space then the automorphisms of $\mathcal{CA}_p$ form a finite dimensional vector space thus implying finite dimensionality of local contact mappings.

The integrability condition can be used to construct contact vector fields, however it often requires solving a large system of nonlinear PDE's which is usually a difficult task, even with software such as MAPLE. In  \cite{tanak1} and \cite{tanak2}, Tanaka developed an algebraic prolongation, which shows that  $\mathcal{CA}_p$  is  finite dimensional if and only if the prolongation is finite dimensional. Furthermore, Tanaka showed that finiteness of his prolongation is equivalent to the finiteness of the usual prolongation of the subalgebra $\Lie{h}_0(\Lie{n})$ consisting of the derivations of $\Lie{n}$ which annihilate all except the first stratum. Note here that  the usual prolongation refers to the prolongation in the sense of Singer and Sternberg \cite{SS1}. In \cite{GQS1} and \cite{Spencer} it is shown that infiniteness of the usual prolongation is equivalent to the complexified algebra containing rank $1$ elements. We thus obtain the following rigidity criteria for $N$. 

\begin{theorem}\label{main1}
Let $N$ be a real Carnot group with Lie algebra $\Lie{n}$. Then the following conditions are equivalent.
\begin{itemize}
 \item[(i)]  $N$ is nonrigid;

 \item[(ii)]  there exists $\varphi_\C$ in the complexification of $\Lie{h}_0(\Lie{n})$  of rank $1$;

 \item[(iii)]   there exists a vector $X_\C$ in the complexification of $\Lie{g}_{-1}$ such that ${\rm ad}X_\C$ has rank $0$ or $1$.
\end{itemize}

\end{theorem}

We remark that \cite{SS1}, \cite{GQS1}, \cite{GQS2} and \cite{Spencer} all precede \cite{tanak1}. Furthermore Remark 2 on page 77 of \cite{tanak1} refers to  \cite{GQS2}, so we would argue that Tanaka was aware of (i) $\Leftrightarrow$ (ii)  in Theorem \ref{main1}, at least at the infinitesimal level, although it is not explicitly stated in \cite{tanak1}.

It is worthwhile to observe that if we substitute conditions (ii) and (iii) with
\begin{itemize}
 \item[(ii)$_\R$] there exists $\varphi$ in  $\Lie{h}_0(\Lie{n})$  of rank $0$ or $1$;

\item[(iii)$_\R$] there exists a vector $X\in\Lie{g}_{-1}$ such that ${\rm ad}X$ has rank $1$,
\end{itemize}

then one still has that (ii)$_\R\Leftrightarrow$ (iii)$_\R$ and they both imply (i). This was observed in~\cite{otz1} and~\cite{OW}.

The main issue in the proof of Theorem~\ref{main1} is (i) $\Rightarrow$ (ii). The first step is  at the infinitesimal level: combining the Tanaka prolongation theory and the result  in~\cite{GQS1} we show that if there are no rank 1 elements in the complexification of  $\Lie{h}_0(\Lie{n})$ then the space of contact vector fields is finite dimensional.
Then we establish the rigidity for $C^2$ contact mappings using a standard integration argument.

 A trivial case of nonrigidity occurs when $\Lie{n}$ is degenerate in the following sense.
We say that $\Lie{n}$ is degenerate if $\Lie{g}_{-1}$ contains nontrivial degenerate elements, that is elements $ X \in \Lie{g}_{-1}$ such that $[X,\Lie{g}_{-1}]=\{0\}$. When  $\Lie{n}$ is degenerate, the corresponding Carnot group is always the direct product with $\R^n$, where $n$ is the dimension of the degenerate space. It follows that the set of local contact mappings contains the local diffeomorphisms of $\R^n$ implying that the corresponding Carnot group is nonrigid.

The Tanaka prolongation theory can be used in order to characterise at the infinitesimal level any space of diffeomorphisms whose differential satisfies some kind of linear condition. A distinguished application is the  generalisation of the Liouville theorem for 1-quasiconformal maps in Carnot groups.
The classical result states that $C^4$--conformal maps between domains of $\R^3$ are the restriction of the action of some element of the group $O(1,4)$. The same result holds in $\R^n$ when $n> 3$ (see, e.g., Nevanlinna~\cite{N}). A major advance in the theory was the passage from smoothness assumptions to metric assumptions (see Gehring~\cite{Geh} and Reshetnyak~\cite{Res}): the conclusion of Liouville's theorem holds for 1-quasiconformal maps. When the ambient space is not Riemannian there are similar theorems. Capogna and Cowling proved in~\cite{CapCow} that 1-quasiconformal maps defined on open subsets of a Carnot group  are smooth and applied this to give some Liouville type results. In particular, they observe that 1-quasiconformal maps between open subsets of H-type groups whose Lie algebra has dimension larger than $2$ form a finite dimensional space. This follows by combining  the smoothness result in~\cite{CapCow} and the work of Reimann~\cite{R}, who established the corresponding infinitesimal result (see also~\cite{CR, CDKR, RR}). Our contribution is the following statement.

\begin{theorem}\label{main2}
The 1-quasiconformal maps between open domains of a Carnot group N, other than $\R$ and $\R^2$, form a finite dimensional space.
\end{theorem}
The proof relies on Theorem 1.1 in~\cite{CapCow} that states that 1-quasiconformal maps are $C^\infty$ and in fact conformal in a suitable sense. This allows us to consider the problem at the infinitesimal level and to use the Tanaka formalism. 

The paper is organised as follows. The first section is dedicated to an overview of the Singer--Sternberg prolongation versus the Tanaka prolongation, and a detailed discussion on the result  in~\cite{GQS1} concerning the rank $1$ condition.
In the second section we prove Theorem~\ref{main1},
and in the third section we prove Theorem~\ref{main2}.

This work stems from discussions during the workshop on Geometry of ODE's and Vector Distributions held at the Stefan Banach
International Mathematical Centre in January 2009. In particular we thank
B. Kruglikov, A. \v Cap, B. Doubrov as well as M. Eastwood for pointing out
the key results in \cite{GQS1} and \cite{Spencer}. We are particularly
grateful to B. Kruglikov for detailed discussions on \cite{GQS1}. We also
point out that in parallel to us, B. Doubrov and O. Radko (formerly O. Kuzmich) also have a manuscript in preparation \cite{DoubRad}, which uses (i) $\Leftrightarrow$ (ii) in Theorem~\ref{main1} to extend the results from \cite{kuzmich} and discuss a number of illustrative examples including a proof that all metabelian Lie algebras with centre of dimension $2$ are nonrigid.

\section{Preliminaries}
\subsection{The Symmetric Algebra}

Let $V={\rm Span}\{e_1,\dots,e_n\}$ be a vector space over the field $\F=\R$ or $\C$, and let $\{\varepsilon_1,\dots,\varepsilon_n\}$ denote the corresponding dual basis for $V^*$. For each integer $k \geq 0$ let $\otimes^k V$ and  $\otimes^k V^*$ denote $k$-th tensor power of $V$ and $V^*$ respectively. Recall that the natural pairing $\langle \, , \,  \rangle:\otimes^k V \times \otimes^k V^* \to \F $ is given by
$$\langle v_1 \otimes \dots \otimes v_k ,  \eta_1 \otimes \dots \otimes \eta_k \rangle =\prod_{i=1}^k\eta_i(v_i).$$

For each integer $k\geq 0$ let  $S^k(V^*)\subset \otimes^k V^*$ denote the symmetric $k$-tensors. An element $K \in S^k(V^*)$ has the form $$K=\sum_{j_1,\dots,j_k=1}^n K_{j_1,\dots,j_k} \varepsilon_{j_1}\otimes \dots \otimes \varepsilon_{j_k} $$ where the coefficients $K_{j_1,\dots,j_k}\in \F$ are
symmetric with respect to permutation of the indices $j_1,\dots,j_k$. The symmetric product of elements of $V^*$ is defined linearly by the symmetrization operator, i.e.,
\begin{align*}
\varepsilon_{j_1}\odot \dots \odot \varepsilon_{j_k} = {\rm Sym}(\varepsilon_{j_1}\otimes \dots \otimes \varepsilon_{j_k}) =\frac{1}{k!}\sum_{\sigma \in S_k}\varepsilon_{j_{\sigma(1)}}\otimes \dots \otimes \varepsilon_{j_{\sigma(k)}},
\end{align*}
and the same definitions apply to elements of $V$. Note that $K \in S^k(V^*)$ if and
only if $ {\rm Sym}(K)=K$, moreover
$$K=\sum_{1 \leq j_1 \leq \dots \leq j_k} k!\,K_{j_1,\dots,j_k}\, \varepsilon_{j_1}\odot \dots\odot \varepsilon_{j_k}. $$

The natural pairing restricted to $S^k(V)\otimes S^k(V^*)\to \F$ gives
\begin{align}
\langle v_1 \odot \dots \odot v_k, \eta_1 \odot \dots \odot \eta_k \rangle=\frac{1}{k!}\sum_\sigma \prod_{i=1}^k \eta_{\sigma(i)}(v_i) \label{Npair}
\end{align}
which is nondegenerate and thus identifies $S^k(V)^*$ with $S^k(V^*)$.

 We obtain a symmetric multilinear map $V^k \to\F$, also denoted by $K$, by setting
\begin{align}
K(v_1,\dots,v_k) &=\langle v_1 \otimes \dots \otimes v_k ,  K \rangle \nonumber \\
&=  \sum_{j_1,\dots,j_k=1}^nK_{j_1,\dots,j_k} \langle v_1 \otimes \dots \otimes v_k,  \varepsilon_{j_1}\otimes \dots \otimes \varepsilon_{j_k} \rangle \nonumber \\
&= \sum_{j_1,\dots,j_k=1}^nK_{j_1,\dots,j_k}
\varepsilon_{j_1}(v_1)\varepsilon_{j_2}(v_2)\dots \varepsilon_{j_k}(v_k). \label{KSymtens}
\end{align}
Since a symmetric multilinear map $K:V^k \to\F$ always has the form given by
(\ref{KSymtens}), we conclude that the space of symmetric multilinear maps identifies with $S^k(V^*)$. Furthermore, restricting $K$ to the diagonal of $V^k$ gives rise to a degree $k$ homogeneous polynomial on $V$ which we also denote by $K$ and we write
$$
K(v)=K(v,\dots,v)=\langle K, v^k  \rangle.
$$
Moreover, we can recover $K(v_1,\dots,v_{k})$ as the coefficient of $\lambda_1\lambda_2 \dots \lambda_{k}$ in the expression $K(\lambda_1v_1 +\dots+ \lambda_{k}v_{k}),$
and thus we have an identification of $k$-homogeneous polynomials on $V$ with $S^k(V^*)$. Similarly  $k$-homogeneous polynomials on $V^*$ identify with $S^k(V)$ where $P \in S^k(V)$ is identified with the polynomial $P(\eta)=\langle P, \eta^k \rangle$. The passage from a polynomial to a symmetric multilinear map is called {\it polarisation}.

The set of symmetric multilinear maps $V^k \to V$ is isomorphic with $V\otimes S^k(V^*)$. An element  $L \in V \otimes S^k(V^*)$ has the form
\begin{align}
L&=\sum_{i=1}^n \sum_{j_1,\dots,j_k=1}^n L_{j_1,\dots,j_k}^i e_i \otimes \varepsilon_{j_1}\otimes \dots \otimes \varepsilon_{j_k} \label{symtens}
\end{align} where $L_{j_1,\dots,j_k}^i$is symmetric in the lower indices and the corresponding multilinear map, also denoted by $L$, has the form
\begin{align} L(v_1,\dots,v_k)&=\sum_{i=1}^n \Big (\sum_{j_1,\dots,j_k=1}^n L_{j_1,\dots,j_k}^i \varepsilon_{j_1}(v_1) \dots  \varepsilon_{j_k}(v_k)\Big ) e_i. \label{SymLinMap}
\end{align}
Since every  symmetric multilinear map $V^k \to V$ has the form given by (\ref{SymLinMap}),  we conclude that this space of maps identifies with  $V\otimes S^k(V^*)$. Furthermore, interpreting $S^k(V^*)$ as the set of $k$-homogeneous polynomials on $V$ we can view $V\otimes S^k(V^*)$ as the space of  $k$-homogeneous polynomial maps $V \to V$. Obviously the same holds for $V^*\otimes S^k(V)$.
Furthermore, we note that (\ref{symtens}) can be written in the form
\begin{align*}
L=\sum_{j_2,\dots,j_k=1}^n \Big(\sum_{i, j_1=1}^n L_{j_1,\dots,j_k}^i e_i \otimes \varepsilon_{j_1} \Big)\otimes \varepsilon_{j_2} \otimes \dots \otimes \varepsilon_{j_k},
\end{align*}
and the expression $A_{j_2,\dots,j_k} =\sum_{i, j_1=1}^n L_{j_1,\dots,j_k}^i e_i \otimes \varepsilon_{j_1}$ is an element of $\Lie{gl}(V)$ identifying with the matrix $\big (L_{j_1,\dots,j_k}^i\big  )_{i,j_1} \in \Lie{gl}(\F^n)$. Since  $A_{j_2,\dots,j_k}$ is symmetric under permutation of the indices ${j_2,\dots,j_k}$,  it follows that $ V \otimes S^k(V^*)$ embeds into $\Lie{gl}(\F^n) \otimes S^{k-1}(V^*)$.

The symmetric algebra is given by
\begin{align}
S(V^*)= \bigoplus_{k \geq 0} S^k(V^*) =\F \oplus V^* \oplus S^2(V^*) \oplus S^3(V^*) \oplus \cdots, \label{symAlg}
\end{align}
and elements $T\in S(V^*)$ have the form $T=\sum_k T^k$, where $T^k \in S^k(V^*)$ and $T^k=0$ for all but finitely many $k \in \N$.

\subsection{Singer--Sternberg Prolongation}\label{SSprol}


Let $\Lie{gl}(V)$ denote the Lie algebra of linear endomorphisms of  $V$, and let  $\Lie{g}^{(0)}$ be a Lie subalgebra of  $\Lie{gl}(V)$. Then for each nonnegative integer $k$, the {\it $k$-th Singer--Sternberg prolongation} $\Lie{g}^{(k)}$ of $\Lie{g}^{(0)}$ is the space  $$ \Lie{g}^{(k)}=\Lie{g} \otimes S^{k}( V^*)\cap V \otimes S^{k+1}(V^*)  .$$ Each $T \in \Lie{g}^{(k)}$ corresponds with a symmetric multilinear map $T: V^{k+1} \to V $, such that for each $k$--tuple $(v_1,\dots,v_k) \in V^k$, the endomorphism of $V$ given by $v_{k+1} \to T(v_1,\dots,v_k,v_{k+1} )$ is an element of $\Lie{g}$.

The Lie algebra $\Lie{g}^{(0)}$ is said to be of {\it finite type} if $\Lie{g}^{(k)}=\{0\}$ for some $k\geq 0$. Note that it follows that $\Lie{g}^{(k+\ell)}=\{0\}$ for all $\ell$.
In particular, we say that $\Lie{g}^{(0)}$ is of type $k$ if this is the smallest positive integer for which $\Lie{g}^{(k)}=0$.
  Naturally $\Lie{g}^{(0)}$ is said to be of {\it infinite type} otherwise.

The following examples are significant for our purposes, and they can be found in~\cite{Kob}.
\begin{ex} \label{ex1}

Take  $\Lie{g}^{(0)}=\Lie{co}(n)$, where
$$
\Lie{co}(n)=\{A\in\Lie{gl}(n,\R)\,|\, A+A^{tr}=kI\}.
$$
We show that $\Lie{g}^{(0)}$ is of type $2$ if $n\geq 3$.
The first prolongation $\Lie{g}^{(1)}$ is naturally isomorphic to the dual
space $\R^{n*}$ of $\R^n$. Indeed, let $T=(T_{jk}^i)$ be an element of
$\Lie{g}^{(1)}$. Since the kernel of the homomorphism $A\in\Lie{co}(n)
\rightarrow {\rm trace}(A)\in\R$ is precisely
$\Lie{o}(n)=\{A\in\Lie{gl}(n,\R)\,|\, A+A^{tr}=0\}$, and since $\Lie{o}(n)$ is of
type $1$~\cite{Kob}, the linear mapping 
$$
T=(T_{jk}^i)\in\Lie{g}^{(1)}\rightarrow \xi =\left( \frac{1}{n}\sum_i T_{ik}^i \right) \in \R^{n*}
$$
is injective (the kernel is the first prolongation of $\Lie{o}(n)$). To see that this mapping is also surjective, it is enough to observe that $\xi=(\xi_k)$ is the image of $T$ with $T_{jk}^i=\delta^i_j \xi_k + \delta^i_k \xi_j -\delta^j_k \xi_i$. To prove that $\Lie{g}^{(2)}=0$, set now $T=(T^h_{ijk})\in \Lie{g}^{(2)}$. For each fixed $k$, $T^h_{ijk}$ may be considered as the components of an element in $\Lie{g}^{(1)}$ and hence can be uniquely written
$$
T^h_{ijk}=\delta^h_i \xi_{jk} + \delta^h_j \xi_{ik} -\delta^i_j \xi_{hk}.
$$
Since $T^h_{ijk}$ must be symmetric in all lower indices, we have $\sum_h T^h_{hjk}=\sum_h T^h_{hkj}$, from which follows $\xi_{jk}=\xi_{kj}$.
From $\sum_h T^h_{hjk}=\sum_h T^h_{jkh}$, we obtain $(n-2)\xi_{jk}=-\delta_{jk}\sum_h \xi_{hh}$, whence $(n-2)\sum_h \xi_{hh}=-n\sum_h \xi_{hh}$ and thus $\sum_h \xi_{hh}=0$. Therefore $(n-2)\xi_{jk}=-\delta_{jk}\sum_h \xi_{hh}=0$, which implies $\xi_{jk}=0$ if $n\geq 3$.
\end{ex}

\begin{ex} \label{ex2}
Suppose $\Lie{g}^{(0)}$ contains a linear endomorphism $T$ of rank $1$, then $T=v_0 \otimes \omega$ where $v_0 \in V$ and $\omega \in V^*$. For any positive integer $k$ we can define a symmetric  $(k+1)$-linear map $T_k: V^{k+1} \to V $ by setting $$ T_k(v_1,\dots,v_k,v_{k+1} )=\omega(v_1)\dots \omega(v_{k})\omega(v_{k+1})v_0$$ and the map $$v_{k+1} \to  T_k(v_1,\dots,v_k,v_{k+1} )$$ is precisely $\omega(v_1)\dots \omega(v_{k})T$ and thus a nontrivial element of $\Lie{g}^{(0)}$. Hence  $\Lie{g}^{(k)}\ne\{0\}$ and $\Lie{g}^{(0)}$ is of infinite type.
\end{ex}

A subalgebra $\Lie{g}^{(0)} \subset \Lie{gl}(V)$ is said to be {\it elliptic} if it does not contain any linear endomorphism of rank $1$, hence finite type subalgebras are elliptic.


\subsection{Complex characteristics}
Throughout this section we consider $V$ to be a vector space over $\C$.
The reason lies on the fact that for our purposes
 we shall make use of the Hilbert's Nullstellensatz.
The goal here is to follow~\cite{GQS1} and  prove that
\begin{align}
&\text{\it the prolongation of $\Lie{g}^{(0)}\subset \Lie{gl}(V)$ is of infinite type if and only if $\Lie{g}^{(0)}$ contains}\nonumber\\&  \text{\it elements of rank $1$}.\label{characteristic}
\end{align}
We begin by observing that
$$V \otimes S^{k+1}(V^*) \subset V \otimes V^* \otimes S^k(V^*)= \Lie{gl}(V)\otimes S^k(V^*)$$
which since  implies that
\begin{align*}
V \otimes S^{k+1}(V^*) &= \Lie{gl}(V)\otimes S^k(V^*) \, \cap \,  V \otimes S^{k+1}(V^*)\\
&\equiv \Lie{g}^{(0)}\otimes S^k(V^*) \, \cap \,  V \otimes S^{k+1}(V^*)\,\\
 &\,\,\,\,\,\,\,\oplus\, \Lie{gl}(V)/\Lie{g}^{(0)} \otimes S^k(V^*) \, \cap \,  V \otimes S^{k+1}(V^*)\\
&= \Lie{g}^{(k)} \, \oplus \,  F^{(k)}
\end{align*}
where
\begin{align*}
 F^{(k)}=V \otimes S^{k+1}(V^*)/ \Lie{g}^{(k)} &= \Lie{gl}(V)/\Lie{g}^{(0)} \otimes S^k(V^*) \, \cap \,  V \otimes S^{k+1}(V^*).
\end{align*}

Observe that if $\Lie{g}^{(0)}$ is of finite type, then for some $k\geq 1$ we have
$$
V \otimes S^{k+1}(V^*) \, \equiv F^{(k)}=  \Lie{gl}(V)/\Lie{g}^{(0)} \otimes S^k(V^*) \, \cap \,  V \otimes S^{k+1}(V^*)
$$
for some $k$ and so every $P^* \in S^{j}(V^*)$, where $j \geq k$, is such that  $\Lie{g}^{(\ell)}\odot P^*\in F^{(l+j)}$ for all $\ell$.  Conversely, if there exists a $k$ such that every $P^* \in S^{j}(V^*)$ satisfies $\Lie{g}^{(\ell)}\odot P^*\in F^{(l+j)}$ for all $\ell$ and $j \geq k$, then  $\Lie{g}^{(0)}$ is of finite type.

We reformulate these observations in the convenient language of exact sequences and consider the problem in the dual. For every integer $k \geq -1$
the following sequence is exact
$$
\Lie{g}^{(k)} \stackrel{\scriptstyle \iota_k} \hookrightarrow V \otimes S^{k+1}( V^*) \stackrel{\scriptstyle \pi_k}{\longrightarrow}   F^{(k)}.
 $$
Here we denoted by $\iota_k$  and $\pi_k$ the canonical immersions and  projections,
and in the case $k =-1$ we put $ \Lie{g}^{(-1)}=V$ and $F^{(-1)}=\{0\}$ . The corresponding  dual sequences take the form
$$
 \Lie{g}^{(k)*} \stackrel{\scriptstyle \iota_k^*}{\longleftarrow}  V^* \otimes S^{k+1}( V)   \stackrel{\scriptstyle \pi_k^*} {\hookleftarrow} F^{(k)*}
$$
where $\pi_k^*(F^{(k)*})$ is the set of all elements in $V^* \otimes S^{k+1}( V)$ which vanish on $ \Lie{g}^{(k)}$, as one can see from standard algebra.
The direct sum of the dual sequences gives the sequence
$$
M^* \stackrel{\scriptstyle \iota^*}{\longleftarrow}   V^* \otimes S(V)   \stackrel{\scriptstyle \pi^*} {\hookleftarrow} F^*,
$$
 where $M^*= \bigoplus_{k \geq -1} \Lie{g}^{(k)*} $, $\iota^*|_{ V^* \otimes S^{k+1}( V)}=\iota_k^*$,  $S(V)=\bigoplus_{k \geq 0}S^k(V) $, $\pi^*|_{F^{(k)*}}=\pi_k^*$ and $F^*= \bigoplus_{k \geq -1} F^{(k)*}$ .

 A nonzero element $\xi \in V^*$ is called a {\it characteristic} if there exists $w \in V$ such that $w \otimes \xi \in \Lie{g}^{(0)}$. Equivalently, $\xi$ is a characteristic if and only if the map $\sigma_\xi:V \to F^{(0)}$ defined by $\sigma_\xi(\cdot)=\pi_0(\,\cdot\, \otimes \xi)$ fails to be injective.
 In fact, note that $w \otimes \xi$ is a rank $1$ element of $\Lie{g}^{(0)}$ and conversely every such element defines a characteristic.
We also note that if $\xi \in V^*$ is not characteristic then we have an injective map $\sigma_\xi :V \to V \otimes V^*/\Lie{g}^{(0)}$ and so necessarily $\dim (V) \leq \dim V \otimes V^*/\Lie{g}^{(0)}$ or equivalently $\dim \Lie{g}^{(0)} \leq n^2-n$. Hence if  $\dim \Lie{g}^{(0)} > n^2-n$ then every element in $V^*$ is characteristic and by Example \ref{ex2} the prolongation of $\Lie{g}^{(0)}$ is infinite.

We say that an element $P \in S(V)$  {\it annihilates} $M^*$ if
$$
V^* \otimes S (V) \odot P \subset \pi^*(F^*).
$$
We observe that such a $P$ cannot have constant term and in fact must lie in $S^+(V)=\oplus_{\ell > 0} S^\ell(V)$ . The set {\scriptsize\calligra I}\, \,  consisting of all annihilating $P$ is an ideal in the ring $S(V)$. Moreover, the following two lemmas show that if  {\scriptsize\calligra V}\, \, denotes the set consisting of all characteristics,  then the set of common zeros of   {\scriptsize\calligra I}\,\, is exactly \,\, {\scriptsize\calligra V}\,\,$ \  \cup \ \{0\}$ .

\begin{lemma} If  $P \in$  {\scriptsize\calligra I}\,\,, then $P(\xi)=0$ for all $\xi \in ${\scriptsize\calligra V}\,\,. \end{lemma}

\begin{proof}

Let $P_{\ell}$ be a homogeneous component of degree $\ell$. Define a linear map $\alpha_{P_{\ell}}^*:V^* \to F^{({\ell-1})*}$ as follows: the annihilating assumption implies $\varepsilon_i \otimes P_{\ell} =\pi_{{\ell-1}}^*(B_i)$ for some $B_i \in F^{({\ell-1})*}$, so that we set $ \alpha_{P_{\ell}}^*({\varepsilon_i}^*)=B_i$. If $v^*=\sum_i a_i \varepsilon_i$ then
 $$
\pi_{{\ell-1}}^* \circ \alpha_{P_{\ell}}^*(v^*)=\pi_{{\ell-1}}^* \circ \sum_i a_i \alpha_{P_{\ell}}^*({\varepsilon_i}^*)=\sum_i a_i \pi_{{\ell-1}}^* ( B_i)=\sum_i a_i\varepsilon_i \otimes P_{\ell}=v^* \otimes  P_{\ell} ,
$$
 thus it follows that $\pi_{{\ell-1}}^* \circ \alpha_{P_{\ell}}^*$ is a right tensor multiplication by $P_{\ell}$.

If $\xi$ is a characteristic then  $\pi_{\ell-1}(w\otimes \xi^{\ell})=0$ for some $w\in V$ and for all $\ell \geq 1$, whence
\begin{align*}
0&=\langle \alpha_{P_{\ell}} \circ \pi_{\ell-1} (w \otimes \xi^{\ell}), w^*\rangle\\
 &=\langle w \otimes \xi^{\ell} ,\pi_{\ell-1}^* \circ
\alpha_{P_{\ell}}^*(w^*)\rangle\\
&=\langle w \otimes \xi^{\ell},w^* \otimes P_{\ell} \rangle\\
 &=\langle w,w^*\rangle \langle \xi^{\ell}, P_{\ell} \rangle,
\end{align*}
and so $\langle \xi^{\ell}, P_{\ell} \rangle=0$.
Hence if $P=\sum_{\ell \geq 1} P_{\ell}$ annihilates $M^*$, then $P(\xi)=0$ for all characteristics $\xi$.
\end{proof}

Let now $\xi$ be not characteristic so that $\sigma_\xi(\cdot) = \pi_0(\,\cdot\,\otimes \xi)$ defines a linear injection of $V$ into $F^{(0)}$. Denote by $B_\xi:F^{(0)} \to V$ the obvious linear extension to $F^{(0)}$ of the left inverse of $\sigma_\xi$. For every $\eta \in V^*$ set  $P_\xi(\eta)=\det (B_\xi \circ \sigma_\eta)$ where  $\sigma_\eta(\cdot) = \pi_0(\,\cdot\, \otimes \eta)$. Then we have $P_\xi(\xi)=1$ and $P_\xi \in S^n(V)$ where $n=\dim V$. Moreover, we prove the following.

\begin{lemma} $P_\xi \in$ {\scriptsize\calligra I}  \end{lemma}

\begin{proof}

We begin by observing that the map $\eta \to B_\xi \circ \sigma_\eta$ is linear and so $$B_\xi \circ \sigma_\eta=\sum_\ell \eta({\textstyle e_\ell})B_\xi \circ \sigma_{\varepsilon_\ell}=\sum_\ell \eta({\textstyle e_\ell})Q^\ell_\xi. $$

For every $\eta \in V^*$ let $L_{\eta,\xi}$ denote the classical adjoint or adjugate of $B_\xi \circ \sigma_\eta$, that is $L_{\eta,\xi} \circ B_\xi \circ \sigma_\eta (v) =P_\xi(\eta)v.$ If $A(j|i)$ denotes the matrix obtained by deleting the $j$-th row and $i$-th column of $A$, then we have the explicit formulae 
$$
[L_{\eta,\xi}]_{ij}=(-1)^{i+j}\det (B_\xi \circ \sigma_\eta(j|i)) =(-1)^{i+j}\det\sum_\ell \eta(e_\ell)Q^\ell_\xi(j|i).
$$
 Since  $\eta({\textstyle e_1})^{i_1} \dots \eta({\textstyle e_n})^{i_n} = \langle  e^I, \eta^{|I|}\rangle $ where $I=(i_1,\dots, i_n)$ and $|I|=\sum_\ell i_\ell$,  it follows that $L_{\eta,\xi}$ has the form
$$L_{\eta,\xi}= \sum_{|I|= n-1} \langle  e^I, \eta^{|I|}\rangle \, Y_I(\xi)$$ where the matrices $ Y_I(\xi)$ do not depend on $\eta$.

If $v \in V$ and $\omega \in V^*$ are arbitrary, then it follows that
\begin{align*}
\langle v \otimes  \eta^{n}, \omega  \otimes  P_\xi \rangle&= \langle P_\xi(\eta)v  , \omega   \rangle = \langle L_{\eta,\xi} \circ B_\xi \circ \sigma_\eta (v)\,  ,\, \omega   \rangle \\
&= \sum_{|I|= n-1} \langle e^I, \eta^{n-1}\rangle \big \langle Y_I(\xi) \circ B_\xi \circ \sigma_\eta (v)\, , \, \omega  \big  \rangle \\
&= \sum_{|I|= n-1}    \langle e^I, \eta^{n-1}\rangle \big  \langle  \sigma_\eta (v)\, ,\, B_\xi^*  \circ Y_I(\xi)^* (\omega)  \big  \rangle \\
&= \sum_{|I|= n-1}  \langle e^I, \eta^{n-1}\rangle \big \langle  \pi_0(v \otimes \eta)\, ,\, B_\xi^*  \circ Y_I(\xi)^* (\omega)  \big  \rangle \\
&=\sum_{|I|= n-1}  \langle e^I, \eta^{n-1}\rangle \big \langle  v \otimes \eta\, ,\, \pi_0^*(B_\xi^*  \circ Y_I(\xi)^* (\omega) ) \big  \rangle \\
&= \sum_{|I|= n-1}  \big \langle  v \otimes \eta^n \, ,\, \pi_0^*(B_\xi^*  \circ Y_I(\xi)^* (\omega)) \otimes e^I \big  \rangle\\
&=  \big \langle  v \otimes \eta^n \, ,\,  \sum_{|I|= n-1} \pi_0^* (B_\xi^*  \circ Y_I(\xi)^* (\omega)) \otimes e^I \big  \rangle.
\end{align*}
It now follows by polarisation that
  $$
\omega  \otimes  P_\xi  =  \sum_{|I|= n-1} \pi_0^*(B_\xi^*  \circ Y_I(\xi)^* (\omega)) \otimes e^I
$$ 
and, since $B_\xi^*  \circ Y_I(\xi)^* (\omega) \in F^{(0)*}$, we conclude that 
$$
\omega  \otimes  P_\xi \, \in \, \pi_0^*(F^{(0)*}) \otimes S^{n-1}(V) \cap V^* \otimes S^{n}(V)=\pi_{n-1}^*(F^{(n-1)*}).
$$
Finally, observing that
$$
V^* \otimes S(V) \odot P = V^* \otimes P \odot S(V),
$$
it now follows  that $P\in$ {\scriptsize\calligra I}\,\,, as desired.
\end{proof}


We recall the Nullstellensatz.  Let ${\mathbb K}$ be an algebraically closed field and let $R[x_1,\dots,x_n]$ be the ring of polynomials over ${\mathbb K}$ in $n$ indeterminates $x_1,\dots,x_n$.  Given an ideal $J\subseteq R[x_1,\dots,x_n]$ we define
$$
V(J)=\{y \in K^n \ |\ p(y)=0 \;\; {\rm for \; all} \;\; P \in J\}.
$$
Then we denote by $I(V(J))$  the ideal of all polynomials vanishing on  $V(J)$ and we write
$$
\sqrt{J}=\{r\in R\ |\ r^t \in J\ \hbox{for some positive integer}\ t\}
$$
for the radical of $J$.
Hilbert's Nullstellensatz states that  $I(V(J))=\sqrt{J}$ for every ideal $J$.

Back to our situation, let us suppose {\scriptsize\calligra V} $\,=\emptyset$. Then we have  $V(${\scriptsize\calligra I} $\,)=\{0\}$ because  $V(${\scriptsize\calligra I} $\,)=$ \!{\scriptsize\calligra V}\,\, $\, \, \cup \,\,\{0\}$,  and the Nullstellensatz implies that the radical of {\scriptsize\calligra I}\, \,  is $S^+(V)$. In particular, for $t$ sufficiently large, $ (e_j)^t \in $ {\scriptsize\calligra I} \, for all $j=1,\dots,m$. For $k$ large enough, every element of  $S^{k+1}(V)$ is the sum of terms of the form $Q_j \odot \left ({e_j}\right)^t  $ where $Q_j \in S^{k+1-t}(V)$, whence $V^* \otimes S^{k+1}(V) \subseteq  \pi^*(F^{(k)*})$. We then conclude that $\Lie{g}^{(k)*}=\{0\}$ and $\Lie{g}^{(k)}=\{0\}$.

We showed that if $\Lie{g}^{(0)}$ does not contain elements of rank one, then it is of finite type. Conversely, Example~\ref{ex2}
 shows that if $\xi$ is a characteristic and therefore $w\otimes \xi$ is a rank one element of $\Lie{g}^{(0)}$, we obtain  $w\otimes \xi^{\ell+1} \in \Lie{g}^{(\ell)}$ for every $\ell >0$, whence $\Lie{g}^{(0)}$ is of infinite type.

\subsection{Complex vs Real}\label{complexification}

Let $V$ be a real $n$-dimensional vector space and let $V_\C$ denote the complexification of $V$. In this section we prove that $\Lie{g}^{(0)} \in \Lie{gl}(V)$ is of infinite type if and only if the complexified algebra $\Lie{g}^{(0)}_\C =\Lie{g}^{(0)}+i \Lie{g}^{(0)} \in \Lie{gl}(V_\C) $ is of infinite type. Obviously this result implies that $\Lie{g}^{(0)}$ is of finite type if and only if $\Lie{g}^{(0)}_\C$ is of finite type.

Assume $T$ is a nonzero element of $\Lie{g}^{(k)}$.  Given a point $w=(u_1+iv_1,\dots,u_{k+1}+ iv_{k+1} ) \in V_\C^{k+1}$, set $$T_\C(w)= \sum_{r \in I(w)}i^{\mu(r)} T(r )$$ where $$I(w)= \{ r=(r_1, \dots ,r_{k+1}) \, \, | \,\, r_j \in \{u_j,v_j\}\, \}$$ and $\mu(r)=$ the number of $v_j$'s in $r$. It follows that $T_\C(w) \in V_\C \otimes S^{k+1}(V_\C)$, moreover
$$w_1\lrcorner \dots w_k \lrcorner  T_\C(w_{k+1})= \sum_{r \in I(w)}i^{\mu(r)} r_1\lrcorner \dots r_k \lrcorner T(r_{k+1} ) \, \in \, \Lie{g}^{(0)}+i \Lie{g}^{(0)} $$ and we conclude that  $T_\C \in \Lie{g}_\C^{(k)}$.

Assume $T$ is a nonzero element of $ \Lie{g}_\C^{(k)}$. Furthermore, let $\iota:V \hookrightarrow V_\C$ be the inclusion given by $\iota(u)=u+i0$. We note that ${\rm Re} T$ and ${\rm Im}T$ are not complex multilinear maps of $V_\C^{k+1}$, however  ${\rm Re}T \circ \iota$ and ${\rm Im}T\circ \iota$ are real multilinear maps of $V$. For any $k$-tuple $w=(w_1,\dots,w_k)\in V_\C^k$ we also have
\begin{align*}
w_1 \lrcorner \dots w_k \lrcorner {\rm Re}T &= {\rm Re}( w_1 \lrcorner \dots w_k \lrcorner T) \in  \Lie{g}^{(0)},
\end{align*} and it follows that
\begin{align*}
u_1 \lrcorner \dots u_k \lrcorner {\rm Re}T \circ \iota = {\rm Re}( \iota(u_1) \lrcorner \dots \iota(u_k) \lrcorner T ) \in  \Lie{g}^{(0)},
\end{align*}
with the same conclusions applying to ${\rm Im}T$.

Furthermore, $T$ is a complex linear combination of elements in the set $$\{ T(e_{j_1}, \dots ,e_{j_{k+1}})\ | \ e_i=\iota(e_i) \},$$ and so $T \circ \iota$ cannot be identically zero since $T$ is not  identically zero . Hence it follows that at least one of the real multilinear maps $ {\rm Re}T \circ \iota :V \to V$ or $ {\rm Im}T \circ \iota:V \to V$  defines a nonzero element of $\Lie{g}^{(k)}$.

\subsection{Tanaka prolongation of stratified nilpotent Lie algebras}\label{Tan}

Consider  a stratified nilpotent Lie algebra $\Lie{n}=\Lie{g}_{-s} \oplus \dots \oplus \Lie{g}_{-1} $.  The Tanaka prolongation of $\Lie{n}$ is the graded Lie algebra $\Lie{g}(\Lie{n})$ given by the direct sum $\Lie{g}(\Lie{n})= \bigoplus_{k \in \Z} \Lie{g}_k(\Lie{n}), $ where $ \Lie{g}_k(\Lie{n})=\{0\}$ for $k < -s$, $\Lie{g}_k(\Lie{n})=\Lie{g}_k$ for $-s \leq k \leq -1$, and for each $k \geq 0$, $\Lie{g}_k(\Lie{n})$ is inductively defined by
$$\Lie{g}_k (\Lie{n})=\Big \{  \varphi \in \bigoplus_{p<0} \Lie{g}_{p+k}(\Lie{n}) \otimes \Lie{g}_p^* \ | \ \varphi([X,Y])=[\varphi(X),Y]+[X,\varphi(Y)] \Big \},$$
with $\Lie{g}_0(\Lie{n})$ consisting of the strata preserving derivations of $\Lie{n}$. If  $\varphi \in \Lie{g}_k(\Lie{n})$, where $k \geq 0$, then the condition in the definition becomes the Jacobi identity upon setting $[\varphi,X]=\varphi(X)$ when $X \in \Lie{n}$. Furthermore, if $\varphi \in \Lie{g}_k(\Lie{n})$ and $\varphi^\prime \in \Lie{g}_\ell(\Lie{n})$, where $k,\ell \geq 0$, then $[\varphi,\varphi^\prime] \in \Lie{g}_{k+\ell}(\Lie{n})$ is defined inductively according to the Jacobi identity, that is
$$ [\varphi,\varphi^\prime](X)=[\varphi,[\varphi^\prime,X]]-[\varphi^\prime,[\varphi,X]].$$

In \cite{tanak1}, Tanaka shows that $\Lie{g}(\Lie{n})$ determines the structure of the contact vector fields on the group $N$ with Lie algebra $\Lie{n}$. In particular, there is an isomorphism between the set of contact vector fields and $\Lie{g}(\Lie{n}) $ when the latter is finite dimensional. Also, $\Lie{g}(\Lie{n}) $ is finite dimensional if and only if  $\Lie{g}_k(\Lie{n})=\{0\}$ for some $k \geq 0$ since  $\Lie{g}_k(\Lie{n})=\{0\}$ implies  $\Lie{g}_{k+\ell}(\Lie{n})=\{0\}$ for all $\ell >0$. Hence the rigidity of a stratified nilpotent Lie group can be determined by studying the Tanaka prolongation of the Lie algebra.

For a given subalgebra $\Lie{g}_0 \subset \Lie{g}_0(\Lie{n})$, the prolongation of the pair $(\Lie{n},\Lie{g}_0)$ is defined as the graded subalgebra
$${\rm Prol}(\Lie{n},\Lie{g}_0)= \Lie{n} \oplus \bigoplus_{k \geq 0} \Lie{g}_k \subset \Lie{g}(\Lie{n}),$$ where for each $k \geq 1$, $\Lie{g}_k$ is inductively defined as the subspace of $ \Lie{g}_k(\Lie{n})$ satisfying $[\Lie{g}_k,\Lie{g}_{-1}] \subseteq \Lie{g}_{k-1}$.  The pair $(\Lie{n},\Lie{g}_0)$ is said to be of finite type if $\Lie{g}_k=\{0\}$ for some $k$, otherwise it is of infinite type and $\Lie{g}(\Lie{n})$ is infinite dimensional.

The type of $(\Lie{n},\Lie{g}_0)$ is determined by the subalgebra $$\Lie{h}= \bigoplus_{k \geq -1} \Lie{h}_k \subset {\rm Prol}(\Lie{n},\Lie{g}_0)$$ where the subspaces $\Lie{h}_k \subset \Lie{g}_k$ are defined as follows: let $$\hat {\Lie{n}}=\Lie{g}_{-s}\oplus \dots\oplus \Lie{g}_{-2} $$ and for $k \geq -1$ define
$$\Lie{h}_k=\left \{  u \in \Lie{g}_k \, | \, [u,
\hat{\Lie{n}}]=\{0\}  \right\}.$$
It follows that $[\Lie{h}_k,\Lie{g}_{-1}] \subset \Lie{h}_{k-1}$ for $k \geq 0$. The bracket generating property shows that $\Lie{h}_0$ identifies with a subalgebra $\Lie{h}^{(0)} \subseteq  \Lie{gl}(\Lie{g}_{-1})$, moreover $\Lie{h}_k$ identifies with  $\Lie{h}^{(k)}$ and by \cite[Corollary 2, page~76]{tanak1} $(\Lie{n},\Lie{g}_0)$ is of infinite type if and only if $\Lie{h}^{(0)}$ is of infinite type in the usual sense.
We extend the notation by writing $\Lie{h}_k(\Lie{n})$ and $\Lie{h}^{(k)}(\Lie{n})$ when $\Lie{g}_0=\Lie{g}_0(\Lie{n})$, and summarise Tanaka's result as follows.
\begin{theorem}\cite[Corollary 2, page~76]{tanak1}\label{Tanaka}
 The space of $C^\infty$ contact vector fields on $N$ is finite dimensional if and only if $\Lie{h}^{(0)}(\Lie{n})$ is of finite type in the usual sense.
\end{theorem}

\subsection{Tanaka prolongation of contact vector fields}\label{prolfield}

In this section we focus on how the space of contact vector fields on $N$ is related to $\Lie{g}(\Lie{n})$, the Tanaka prolongation algebra through $\Lie{g}_0(\Lie{n})$.
We shall need this construction in Section~\ref{Liouville}, where we investigate the conformal vector fields which form a subalgebra of the Lie algebra of contact vector fields.
In particular, we shall see that the Lie algebra of conformal vector fields is isomorphic to the Tanaka prolongation through a suitable subalgebra of $\Lie{g}_0(\Lie{n})$.

The argument we propose here follows \cite[Section 2]{yam}. Given  a contact vector
field $U$ on $N$, we shall  define correspondingly a section of the bundle $N \times
\Lie{g}(\Lie{n}) \to N$.

Set $\Lie{n}=\Span{X_{-i,k}\ | \ i=1,\dots,s, \ \  k=1,\dots,d_i}$, where $s$ is the step of $\Lie{n}$ and $d_i$ is the dimension of $\Lie{g}_{-i}$.
Then the Lie algebra of left invariant fields
is spanned by the vector fields  $\tilde X_{-i,k}=(\tau_p)_*(X_{-i,k})$. Using this basis we write any two vector fields $U$ and $W$ as
$$
U=\sum_{i,k}u_{-i,k}\tilde X_{-i,k}\qquad W=\sum_{j,l}w_{-j,l}\tilde X_{-j,l}.
$$

We define the $\Lie{n}$--valued function $A_U$ as
$$
A_U(p) =\sum_{i,k}u_{-i,k}(p)X_{-i,k}
$$
and $A_W$ accordingly. We then have the equation
\begin{align}
A_{[U,W]}(p)&= [A_U(p),A_W(p) ]+U(A_W)(p)-W(A_U)(p), \label{Aident}
\end{align} where for example $U(A_W)(p)=\sum_{i,k}(Uw_{-i,k})(p)X_{-i,k}$.

If $U$ is a contact vector field, and $W(p)=\tilde X_{-1}(p)=(\tau_p)_*(X_{-1})$ where $X_{-1} \in \Lie{g}_{-1}$, then $A_{\tilde X}(p)=X_{-1}$ and (\ref{Aident}) becomes 
\begin{align*}
A_{[U,\tilde X_{-1}]}(p)&= [A_U(p),X_{-1} ]-\tilde X(A_U)(p).
\end{align*}
Since $A_{[U,\tilde X_{-1}]}(p) \in \Lie{g}_{-1}$, the previous equality implies that
\begin{align}
 \tilde X_{-1}(A_U^{-m})(p)=[A_U^{-m+1}(p), X_{-1}]\label{cntct2}
\end{align}for all $m>1$.
If $\tilde Y_{-1} $ is another horizontal
left invariant vector field, then (\ref{cntct2}) implies
\begin{align*}
 \tilde Y_{-1} \tilde X_{-1}(A_U^{-m})(p)&=[\tilde
Y_{-1}(A_U^{-m+1})(p), X_{-1}] \\
&=[[A_U^{-m+2}(p), Y_{-1}], X_{-1}]
\end{align*}
for all $m>2$, moreover
\begin{align*}
 [\tilde Y_{-1}, \tilde X_{-1}](A_U^{-m})(p)&=[[[A_U^{-m+2}(p), Y_{-1}], X_{-1}] -  [[A_U^{-m+2}(p), X_{-1}], Y_{-1}] \\
&=[ A_U^{-m+2}(p),[ Y_{-1},X_{-1}]]
\end{align*} for all $m>2$.
By the bracket generating property we conclude that
\begin{align}
\tilde X_{-r}(A_U^{-m})(p) &=[ A_U^{-m+r}(p), X_{-r}],  \quad
\forall   X_{-r} \in  {\Lie{g}}_{-r} \quad {\rm s.t.}
\quad r<m. \label{yamstuff1}
\end{align}
Define now $A_U^0:N \to \bigoplus_{\ell<0}  {\Lie{g}}_\ell(\Lie{n}) \otimes
 {\Lie{g}}_\ell^*$  by setting
\begin{align*}
A_U^0(p)( X_{-r})=\tilde X_{-r} (A_U^{-r})(p), \quad 
X_{-r} \in  {\Lie{g}}_{-r}.
\end{align*}
For $ X_{-r} \in  {\Lie{g}}_{-r}$ and
$Y_{-t} \in  {\Lie{g}}_{-t}$ we have
\begin{align*}
A_U^0(p)([Y_{-t}, X_{-r}])
&=[\tilde Y_{-t},\tilde X_{-r}] (A_U^{-(r+t)})(p)\\
&=\tilde Y_{-t} \tilde X_{-r} (A_U^{-(r+t)})(p)-\tilde X_{-r} \tilde
Y_{-t} (A_U^{-(r+t)})(p)\\
&=[\tilde Y_{-t} (A_U^{-t})(p),  X_{-r}]-[\tilde X_{-r}
(A_U^{-r})(p), Y_{-t}]\\
&=[ A_U^{0}(p)( Y_{-t}),  X_{-r}]-[ A_U^{0}(p)(
X_{-r}), Y_{-t}]\\
&=[ A_U^{0}(p)( Y_{-t}),  X_{-r}]+[ 
Y_{-t},A_U^{0}(p)( X_{-r} )],
\end{align*}  that is $A_U^0(p)$ is a strata preserving
derivation of $ {\Lie{n}}$. Furthermore, the previous
equality is simply the Jacobi identity when we define
\begin{align*}
[A_U^0(p), X]:=A_U^0(p)( X)
\end{align*}
for any  $X\in\Lie{n}$. In summary $A_U^0(p) \in
{\Lie{g}}_0(\Lie{n})$ where $$ {\Lie{g}}_0(\Lie{n}) = \Big \{ \varphi \in
\bigoplus_{\ell<0}  {\Lie{g}}_\ell (\Lie{n}) \otimes  
{\Lie{g}}_\ell^*\ | \  \varphi([ Y,  X])=  [\varphi(
Y),  X])+[ Y,  \varphi( X)])\Big \}.$$
Inductively on $n=0,1,2,\dots$, we define  $A_U^n:N \to \bigoplus_{\ell<0}
{\Lie{g}}_{\ell+n}(\Lie{n})  \otimes   {\Lie{g}}_\ell^*$  by
setting
\begin{align}
A_U^n(p)( X_{-r})= \tilde X_{-r} (A_U^{-r+n})(p) \quad \text{when
}  X_{-r} \in {\Lie{g}}_{-r}. \label{yamstuff2}
\end{align} 
At each step of the inductive construction, for the case
$r=1$, we define $\tilde X (U_{-1+n})(p)$  relative to the basis of 
${\Lie{g}}_{n-1}(\Lie{n})$ induced by the chosen basis of $\Lie{n}$ combined with the corresponding dual basis.  It follows that
\begin{align*}
A_U^n(p)([ Y_{-t}, X_{-r}])&=[\tilde Y_{-t},\tilde X_{-r}] (A_U^{-(r+t)+n})(p)\\
&=\tilde Y_{-t} \tilde X_{-r} (A_U^{-(r+t)+n})(p)-\tilde X_{-r} \tilde
Y_{-t} (A_U^{-(r+t)+n})(p)\\
&=[\tilde Y_{-t} (A_U^{-t+n})(p),  X_{-r}]-[\tilde X_{-r}
(A_U^{-r+n})(p), Y_{-t}]\\
&=[ A_U^{n}(p)( Y_{-t}),  X_{-r}]-[ A_U^{n}(p)(
X_{-r} ), Y_{-t}]\\
&=[ A_U^{n}(p)( Y_{-t}),  X_{-r}]+[ 
Y_{-t},A_U^{n}(p)( X_{-r} )],
\end{align*} which becomes the Jacobi identity when we set
\begin{align}
[A_U^n(p),  X_{-r}]:=A_U^n(p)( X_{-r}).\label{yamstuff3}
\end{align} In summary $ A_U^n(p) \in  {\Lie{g}}_n (\Lie{n})$ where
$$
{\Lie{g}}_n (\Lie{n})= \Big \{ \varphi \in \bigoplus_{\ell<0} 
{\Lie{g}}_{\ell+n} (\Lie{n}) \otimes   {\Lie{g}}_\ell ^* \ | \
\varphi([ Y,  X])=  [\varphi( Y),  X])+[
Y,  \varphi( X)])\Big \}.
$$ 
For $A_W^m(p) \in 
{\Lie{g}}_m (\Lie{n})$ and $A_U^n(p) \in  {\Lie{g}}_n (\Lie{n})$ we inductively
define $[A_W^m(p),A_U^n(p)] \in  {\Lie{g}}_{m+n} (\Lie{n}) $ according to
the Jacobi identity, that is  for every 
$X\in\Lie{n}$ we set
\begin{align*}
[A_W^m(p),A_U^n(p)](\ X)&=[[A_W^m(p),A_U^n(p)], X]\\
&=[A_W^m(p),[A_U^n(p), X]]-[A_U^n(p),[A_W^m(p), X]].
\end{align*}
Thus for every contact vector field $U$ we obtain
 a section of the bundle $N \times
\Lie{g}(\Lie{n}) \to N$, that is $\sum_{n\geq -s}A_U^n$.

 \section{Proof of Theorem~\ref{main1}}\label{proof1}
{(ii) $\Rightarrow$ (iii).}
Let $\varphi_\C$ be  an element of rank $1$ in the complexification of $\Lie{h}_{0}(\Lie{n})$. Corresponding to $\varphi_\C$  there exists a basis  ${\mathcal B}$ of $\Lie{n}_\C$, the complexification of $\Lie{n}$, and vectors $X_\C,Y_\C\in{\mathcal B}$ in the complexification of $\Lie{g}_{-1}$ (not necessarily distinct), unique up to a scale, such that $\varphi_\C(Y_\C)=X_\C$ and $\varphi_\C(Z)=0$ for every   $Z_\C\in{\mathcal B}\setminus \{Y_\C\}$. For such a $Z_\C$ we have
$$
0=\varphi_\C[Y_\C,Z_\C]=[\varphi_\C Y_\C,Z_\C]+[Y_\C,\varphi_\C Z_\C]=[X_\C,Z_\C]
$$
We conclude that ${\rm rank}\,\ad X \leq 1$.

{(iii) $\Rightarrow$ (ii).}
If $X_\C$ is in the complexification of $\Lie{g}_{-1}$ and $\ad X_\C$ has rank zero, then $\Lie{n}$ is degenerate and it suffices to take the derivation $\varphi_\C$ defined by $\varphi_\C(X)=X$ and zero otherwise.
Suppose now that   $X_\C$ is such that ${\rm rank}\,\ad X_\C =1$.
Following~\cite{otz1},  it is possible to complete $X_\C$ to a basis ${\mathcal B}$ of $\Lie{n}_\C$, in such a way that there exists a unique $Y_\C\in {\mathcal B}$  for which $[X_\C,Y_\C]\neq 0$, moreover $Y_\C$ is in the complexification of $\Lie{g}_{-1}$.
It follows directly that the endomorphism $\varphi_\C:\Lie{n}_\C\rightarrow \Lie{n}_\C$ defined by $\varphi_\C(Y_\C)=X_\C$ and zero elsewhere is an element  of rank $1$ in the complexification of $\Lie{h}_0(\Lie{n})$.

{(ii) $\Rightarrow$ (i).}
Let $\varphi_\C$ be a rank $1$ element in the complexification of $\Lie{h}_0(\Lie{n})$. Then it identifies with an element of rank $1$ in the complexification of $\Lie{h}^{(0)}(\Lie{n})$, say $\Lie{h}^{(0)}_\C(\Lie{n})$. By~\eqref{characteristic} it follows that $\Lie{h}^{(0)}_\C$ is of infinite type, and by the discussion in~\eqref{complexification} $\Lie{h}^{(0)}$ is also of infinite type.
Then Theorem~\ref{Tanaka} implies that $N$ admits infinite dimensional families of $C^\infty$   contact maps, thus $N$ is nonrigid.

{(i) $\Rightarrow$ (ii).}
We show that if condition (ii) fails, then $N$ is rigid in the class of $C^2$ maps.
First, observe that if there is no rank $1$ element in the complexification of $\Lie{h}_0(\Lie{n})$, then arguing as above we conclude that $\Lie{h}^{(0)}$ is of finite type, and so is $\Lie{g}(\Lie{n})$, the Tanaka prolongation of $\Lie{n}$. It follows that the space of contact vector fields on some open domain of $N$ forms a finite dimensional Lie algebra $\mathcal{C}(N)$ which is isomorphic to $\Lie{g}(\Lie{n})$. Denote this isomorphism by $\alpha  : \Lie{g}(\Lie{n})\rightarrow \mathcal{C}(N)$.
In particular, the contact vector fields have polynomial coefficients of uniformly bounded degree.

Let $U$ be a $C^1$ contact vector field, thus satisfying the integrability condition $[U,\mathcal{H}]\subset\mathcal{H}$.
We regularise $U$ by convolving  by a smooth compactly supported function, so that we obtain smooth approximations $U_n$ to $U$ which still satisfy the integrability condition. Then the vector fields $U_n$ are all polynomial of uniformly bounded degree, and approximate $U$ in the $C^1$ topology. It then follows that $U$ is polynomial and in the $C^\infty$ class.

Finally, take $f$ a $C^2$ contact map between open domains of $N$. By composing with left translations, we can assume that $f$ fixes the identity, namely $f(e)=e$. If $U$ is in $\alpha(\Lie{g}(\Lie{n}))$, then $f_* U$ is $C^1$, and
$$
[f_*U,\mathcal{H}]=f_*[U,f_*^{-1}\mathcal{H}]\subset f_*[U,\mathcal{H}]\subset f_*\mathcal{H}\subset \mathcal{H},
$$
whence $f_*U$ is a contact vector field and therefore it determines a unique element in $\alpha (\Lie{g}(\Lie{n}))$.
So $\alpha^{-1}f_* \alpha$ is an automorphism of $\Lie{g}(\Lie{n})$. Therefore we conclude that $f$ is induced by an automorphism of $\Lie{g}(\Lie{n})$ and so it varies in a finite dimensional space.

\section{The Liouville theorem}\label{Liouville}
In this section we prove Theorem~\ref{main2}.

Let $\mathcal{U},\mathcal{V}\subset N$ be open domains, and let $f :\mathcal{U}\rightarrow \mathcal{V}$  be a homeomorphism. For $p\in \mathcal{U}$ and for small $t\in \R$ we define the distortion as
$$
H_{f}(p,t)= \frac{{\rm max}\{d(f(p),f(q)) | d(p,q)=t\}}{{\rm min}\{d(f(p),f(q))|d(p,q)=t\}}.
$$
We say that $f$ is $K$-quasiconformal if there exists a constant $K \geq 1$ such that $$\lim \sup_{t\rightarrow 0} H_f (p,t)\leq K$$ for all $p\in \mathcal{U}$.

Quasiconformal maps are Pansu differentiable, see \cite{pansu1}, which implies they satisfy the contact conditions almost everywhere. We recall that  the Pansu differential $Df(p)$ is  the automorphism of $N$ defined as
$$
Df(p)\,q = \lim_{t\rightarrow 0} \delta_t^{-1}\circ \tau_{f(p)}^{-1}\circ f \circ \tau_p\circ \delta_t(q),
$$
where $p,q\in N$. The Lie derivative of $Df(p)$, say $df(p)$, is a grading-preserving automorphism of $\Lie{n}$. Suitably regular contact maps are Pansu differentiable. In particular $C^1$ contact maps are Pansu differentiable everywhere, see \cite{warhurst2}, and more generally almost everywhere Pansu differentiable when locally Lipschitz, see \cite{magn}.

In~\cite{CapCow}, the authors prove that $1$-quasiconformal maps are $C^\infty$ and characterised by the conditions of being locally quasiconformal and $Df(p)$ coinciding with a similarity, i.e. a product of a dilation and an isometry of the subriemannian metric. This implies in particular that $df(p)$ is a similarity when restricted to $\Lie{g}_{-1}$~\cite[Lemma 5.2]{CapCow}.
The tangent mapping $f_*$ can be represented in invariant coordinates by the matrix  $M_f(p) = ( \tau_{f(p)}^{-1} \circ f \circ \tau_p)_*|_e$, which
 coincides with $df(p)$ when restricted to $\Lie{g}_{-1}$.
 More explicitly, for all $X,Y \in \Lie{g}_{-1}$ we have
\begin{align*}
\langle  M_f(p) (X), M_f(p) (Y) \rangle &=\lambda(p) \langle X,Y \rangle
\end{align*}
and so  for every $p\in\mathcal{U}$, the restriction of  $f_*(p)$ to  the horizontal space $\mathcal{H}_p$  is given by an element in the conformal group $$CO(n)= \{M \in GL(n,\R) \,|\, M^{tr}M= \lambda I\},$$ where  $n=\dim \Lie{g}_{-1}$. In view of this fact we shall also refer to 1-quasiconformal maps as conformal maps.

Let  $f_t$ be a 1--parameter group of conformal maps defined on an open set $\mathcal{U}$, and let $U$ be the contact vector field which generates $f_t$. If $\tilde X$ is the left invariant vector field corresponding to a horizontal vector $X$, then using the notation of Section~\ref{prolfield} we write $A^0_U(p)(X)=\tilde X (A_U^{-1})(p)$. Denote $\bar M_{f_t}(p)= M_{f_t}(p)|_{\Lie{g}_{-1}}$ and ${\bar A^0_U}(p)={A^0_U}(p)|_{\Lie{g}_{-1}}$.
The following equality
$$\frac{d}{dt}\bar M_{f_t}(p)^{tr}\bar M_{f_t}(p) \Big |_{t=0}=\frac{d}{dt}\lambda_t(p) |_{t=0} I = \lambda'(p)I$$
shows that
$${\bar A^0_U}(p)^{tr}+{\bar A^0_U}(p) = \lambda'(p)I,$$
and we conclude  that the restriction of ${A^0_U}(p)$ to the horizontal space coincides with an element of $$\Lie{co}(n)=\{A\in\Lie{gl}(n,\R)\,|\, A^{tr}+A=\lambda' I \, \, {\rm for \, \, some} \,\, \lambda' \in \R \},$$ where  $n=\dim \Lie{g}_{-1}$.


It follows from~\cite[Lemma 6.3]{tanak1} and \cite[\textsection 2]{yam} that the Lie algebra of conformal vector fields is determined by the Tanaka prolongation on $\Lie{n}$ through
$\Lie{g}_{0}=\{\varphi\in\Lie{g}_0(\Lie{n}): \varphi\in \Lie{co}(n)\}$.
We illustrate with the case of the three dimensional Heisenberg group.

\begin{ex}\label{exHeis}
It is well known~\cite{kr1} that the $2n+1$--dimensional Heisenberg group is nonrigid with respect to its contact structure, whereas the conformal maps coincide with the action of some element in the Lie group $SU(1,n)$. We retrieve the result for the conformal maps at the infinitesimal level by computing the Tanaka prolongation of the conformal vector fields. The notation is that of Section~\ref{prolfield}.

Assume now that $\Lie{n}$ is the three dimensional Heisenberg algebra, and $N$ is the the corresponding Heisenberg group. We fix a basis $\{X_1,X_2,Y\}$ of $\Lie{n}$ such that $[X_1,X_2]=Y$ is the only nonzero bracket and define a scalar product for which $\Lie{g}_{-1}={\rm span}\{X_1,X_2\}$ and $\Lie{g}_{-2}={\rm span}\{Y\}$ are orthogonal.
Choose exponential coordinates on the group $(x_1,x_2,y)={\rm exp}yY{\rm exp}x_2X_2{\rm exp}x_1X_1$. We identify the Lie algebra with the tangent space $T_e N$ to $N$ at the identity $e$, and for $X$ in $\Lie{n}$ we write $\tilde{X}$ for the left--invariant vector field that agrees with $X$ at $e$. The left invariant vector fields corresponding to the basis vectors are
\begin{align*}
\tilde{X}_1 &= \frac{\partial}{\partial x_1}\qquad \tilde{X}_2 = \frac{\partial}{\partial x_2}+x_1\frac{\partial}{\partial y}\\
\tilde{Y}&=\frac{\partial}{\partial y}\,\,.
\end{align*}
Let  $U= f_1\tilde{X}_1 +f_2\tilde{X}_2+g\tilde{Y}$ be a conformal vector field. The contact conditions are
\begin{align*}
\tilde{X}_1 g=-f_2 \quad {\rm and} \quad  \tilde{X}_2 g=f_1
\end{align*}
and the conformality conditions are 
\begin{align*}
\tilde{X}_1f_1 = \tilde{X}_2 f_2  \quad {\rm and} \quad  \tilde{X}_2 f_1 = -\tilde{X}_1 f_2.
\end{align*}
Thus contact and conformality imply
\begin{align}
\tilde{X}_1^2g=\tilde{X}_2^2g  \quad {\rm and} \quad  \tilde{X}_1 \tilde{X}_2 g= -\tilde{X}_2 \tilde{X}_1g. \label{ContConf}
\end{align}
We gather some implications of (\ref{ContConf}) which will be required in the
computations. First we note that the second equation in (\ref{ContConf}) implies
\begin{align}
\tilde Y g =2\tilde X_1\tilde X_2 g  \label{Yg},
\end{align}
and the bracket conditions $[\tilde X_1,[\tilde X_1,\tilde X_2]] =[\tilde X_2,[\tilde X_1,\tilde X_2]]=0$ imply  
\begin{align}
 \tilde X_2 \tilde X_1^2g =-3\tilde X_1^2\tilde X_2 g \qquad  \tilde X_1\tilde X_2^2 g=-3\tilde X_2^2 \tilde X_1g  \label{commutes}.
\end{align}

Let $\{dx_1, dx_2, dy \}$ denote the basis of  $\Lie{n}^*$, dual to the basis we fixed for $\Lie{n}$. For every $p\in \mathcal{U}$ we have
\begin{align*}
A_U^{-1}(p)&= \tilde{X}_2g(p)X_1 - \tilde{X}_1 g(p)  X_2\\
A_U^{-2}(p)&= g(p) Y,
\end{align*}
and for each integer $k \geq 0$ we have 
\begin{align*}
A_U^{k}(p)&=A_U^{k}(p)(X_1)\otimes dx_1+A_U^{k}(p)(X_2)\otimes dx_2
+A_U^{k}(p)(Y)\otimes dy \\
&=\tilde X_1(A_U^{k-1})(p)\otimes dx_1+\tilde X_2(A_U^{k-1})(p)\otimes dx_2
+\tilde Y(A_U^{k-2})(p)\otimes dy.
\end{align*}
A direct calculation shows that
\begin{align*}
A_U^0(p)=&(\tilde Y g)(p)A_1 + (\tilde{X}_1\tilde X_1 g)(p)A_2 
\end{align*}
where
\begin{align*}
A_1 &= \frac{1}{2}{X}_1 \otimes dx_1 + \frac{1}{2}{X}_2 \otimes  dx_2+   Y \otimes dy \\
A_2 &= {X}_1\otimes  dx_2- {X}_2 \otimes dx_1.
\end{align*}
Setting $\Lie{g}_0={\rm span}\, \{
A_1,A_2\}$, we can compute the Lie brackets of $\Lie{n}\oplus \Lie{g}_0$. The
higher order prolongation spaces can then be computed in a purely symbolic way
using the Jacobi identity and the bracket generating property. The advantage in
taking this symbolic approach is that no more information is required from the
contact conditions or the conformality conditions.  
In order to compute $A_U^1(p)$ we use (\ref{ContConf}) and (\ref{commutes}). We obtain
\begin{align}
A_U^1(p)=&  (\tilde Y \tilde X_1 g)(p)B_1+  (\tilde Y \tilde X_2 g)(p) B_2 \label{firstpro}
\end{align}
where
\begin{align*}
B_1&= A_1\otimes dx_1-\frac{3}{2}A_2\otimes dx_2 - \tilde X_2\otimes dy\\
B_2&= A_1\otimes dx_2+ \frac{3}{2}A_2\otimes dx_1 + \tilde X_1\otimes dy.
\end{align*}
In particular, we have used the identities $\tilde X_1^3g=\frac{3}{2}\tilde Y \tilde X_2g$ and $\tilde X_2\tilde X_1^2g=-\frac{3}{2}Y \tilde X_1g$, which are easily derived from (\ref{ContConf}), \eqref{Yg} and (\ref{commutes}).

At the next step of prolongation  we see that
\begin{align*}
A_U^2(p)=&  ( \tilde Y \tilde{X}_1^2 g)(p)C_1 +  (\tilde Y^2 g)(p) C_2
\end{align*}
where
\begin{align*}
C_1&= B_1\otimes dx_1  +B_2\otimes dx_2+ A_2\otimes dy\\
C_2&= \frac{1}{2} B_2 \otimes dx_1 -\frac{1}{2}B_1\otimes dx_2 + A_1\otimes dy.
\end{align*}

We now see that $\tilde Y \tilde{X}_1^2 g(p)=0$ by observing that
\begin{align*}
\tilde Y\tilde X_1^2g&=\tilde X_1\tilde X_2 \tilde X_1^2g-\tilde X_2\tilde X_1 \tilde X_1^2g \nonumber\\
&=\tilde X_1\tilde X_2 \tilde X_1^2g-\tilde X_2\tilde X_1 \tilde X_2^2g  \quad {\rm by \, \,(\ref{ContConf})}\\
&=-3\tilde X_1 \tilde X_1^2\tilde X_2g +3\tilde
X_2\tilde X_2^2\tilde X_1g  \quad {\rm by \, \, (\ref{commutes})} \\
&=-3\tilde X_1^2\tilde X_1\tilde X_2g -3\tilde X_2^2\tilde X_1\tilde X_2g \quad {\rm by \, \,(\ref{ContConf})}\\
&=-\frac{3}{2}\tilde Y( \tilde X_1^2g +\tilde X_2^2 g) \quad {\rm by \, \, (\ref{Yg})} \\
&=-3\tilde Y\tilde X_1^2g \quad {\rm by \, \,(\ref{ContConf})}. 
\end{align*}

Finally, we observe using (\ref{commutes}) that $\tilde Y \tilde{X}_1^2 g=0$ implies
$$\tilde X_1 \tilde Y^2g=\tilde X_2 \tilde Y^2g= Y^3g=0,$$
since
$$
\tilde X_1 \tilde Y^2g=-\frac{2}{3}\tilde X_2 \tilde Y \tilde X_1^2g, \quad \tilde X_2 \tilde Y^2g=\frac{2}{3}\tilde X_1 \tilde Y \tilde X_1^2g \quad {\rm and} \quad  \tilde Y^3g=\frac{4}{3}\tilde X_2^2 \tilde Y \tilde X_1^2g.  
$$ 
It now follows that
\begin{equation}\label{roots}
{\rm Prol}(\Lie{n},\Lie{g}_0)=\Lie{g}_{-2}\oplus\Lie{g}_{-1}\oplus \Lie{g}_0\oplus \Lie{g}_1\oplus \Lie{g}_2
\end{equation}
 where
\begin{align*}
\Lie{g}_0={\rm span}\, \{ A_1,A_2\}, \quad \Lie{g}_1={\rm span}\, \{ B_1,B_2\} \quad {\rm and} \quad   \Lie{g}_2={\rm span}\, \{ C_2\}.
\end{align*}
Setting $H_1=2A_1$, $H_2=A_2$, $\bar{X}_1=2B_1$, $\bar{X}_2=2B_2$ and
$\bar{Y}=-4C_2$, one recognises in~\eqref{roots} the root space decomposition
of the simple Lie algebra $\Lie{su}(1,2)$, and direct calculation according to
the bracket definition in Section \ref{prolfield} gives the
correct bracket relations (see Table~\ref{su12}).

\end{ex}

{\it Proof of Theorem~\ref{main2}.}
By the discussion above, the Lie algebra of conformal vector fields is ${\rm Prol}(\Lie{n},\Lie{g}_0)$. Theorem~\ref{Tanaka} implies that the character of
${\rm Prol}(\Lie{n},\Lie{g}_0)$ is determined by the usual prolongation of $\Lie{h}_{0}(\Lie{n})\cap \Lie{g}_0$. This space identifies with a subalgebra $\Lie{h}^{(0)}$ of $\Lie{co}(n)$.

If $n\geq 3$, it follows from Example~\ref{ex1} that the usual prolongation of $\Lie{co}(n)$ is zero at the second step and so in particular $\Lie{h}^{(2)}=\{0\}$. Thus ${\rm Prol}(\Lie{n},\Lie{g}_0)$ has finite dimension if ${\rm dim}\Lie{g}_{-1}\geq 3$ .

Now set $n=2$. Since $\Lie{n}$ is not abelian, there exist $X_1,X_2$ in $\Lie{g}_{-1}$ such that $[X_1,X_2]=Y\neq 0$. Then it follows that 
$$
\Lie{co}(n)\cap \Lie{h}^{(0)}(\Lie{n})\subseteq \left\{ \bmatrix 0&b\\-b&0\endbmatrix \,|\, b\in\R \right\},
$$
with equality if $\Lie{n}$ is the three dimensional Heisenberg algebra. It follows easily that the space  on the right hand side of the inclusion above is of type $1$ in the sense described in Section~\ref{SSprol}, which proves Theorem~\ref{main2} for the case $n=2$ .
Notice that the same conclusion can of course be reached as a consequence of Example~\ref{exHeis}.

Finally, let $f$ be  a conformal map on some open domain $\mathcal{U}$. Composing with left translation if necessary, we may assume that $f(e)=e$. Arguing as we did at the end of Section~\ref{proof1}, we conclude that $f$ is induced by an automorphism of ${\rm Prol}(\Lie{n},\Lie{g}_0)$, so that the proof of Theorem~\ref{main2} is complete.

\vskip0.3cm
\begin{table}[h!]
\begin{center}
\begin{tabular}{c| c c c c c c c c } 
  & $ Y$ &  $X_1$  & $X_2$ & $H_1$ & $H_2$ & $\bar{X}_1$ & $\bar{X}_2$ & $\bar{Y}$ \\
  \hline
\\
$Y$ &  $0$  & $0$ &  $0$ & $-2Y$ & $0$ & $2X_2$ & $-2X_1$ & $2H_1$\\
\\
$X_1$&$0$&$0$&$Y$&$-X_1$&$X_2$&$-H_1$&$-3H_2$&$\bar{X}_2$\\
\\
$X_2$&$0$&$-Y$&$0$&$-X_2$&$-X_1$&$3H_2$&$-H_1$&$-\bar{X}_1$\\
\\
$H_1$&$2Y$& $X_1$&$X_2$&$0$&$0$&$-\bar{X}_1$&$\bar{X}_2$&$-2\bar{Y}$\\
\\
$H_2$&$0$&$-X_2$&$X_1$&$0$&$0$&$-\bar{X}_2$&$\bar{X}_1$&$0$\\
\\
$\bar{X}_1$&$-2X_2$&$H_1$&$-3H_2$&$\bar{X}_1$&$\bar{X}_2$&$0$&$-2\bar{Y}$&$0$\\
\\
$\bar{X}_2$&$2X_1$&$3H_2$&$H_1$&$\bar{X}_2$&$-\bar{X}_1$&$2\bar{Y}$&$0$&$0$\\
\\
$\bar{Y}$&$-2H_1$&$-\bar{X}_2$&$\bar{X}_1$&$2\bar{Y}$&$0$&$0$&$0$&$0$\\
\\
\end{tabular}
\end{center}
\caption{Bracket table of $\Lie{su}(1,2)$.}
\label{su12}
\end{table}

\end{document}